\documentclass[a4paper,10pt,reqno]{amsart}
\usepackage{amsmath,amsthm,amssymb,mathabx}
\usepackage{mathtools}
\usepackage{xcolor}
\allowdisplaybreaks 

\usepackage
[
  bookmarks = true, 
  pagebackref = true, 
 pdfauthor = {KH}, 
 pdftitle = \text{K Hughes - Reinforcing a philosophy 0}, 
  colorlinks = true, 
  linkcolor = blue, 
  citecolor = blue]
{hyperref}
\newcommand{\C}{\mathbb{C} }
\newcommand{\R}{\mathbb{R} }
\newcommand{\Z}{\mathbb{Z} }

\newcommand{\N}{\mathbb{N} }

\newcommand{\lf}{K} % lf for local field
\newcommand{\roi}{\mathfrak{o}} % roi for ring of integers
\newcommand{\uniformizer}{\pi} 
\newcommand{\partition}{\mathcal{P}} 
\newcommand{\scales}{\mathcal{R}}

\newcommand{\curve}{\gamma} 
\newcommand{\syzygies}{\mathcal{S}} 

\newcommand{\oddprime}{\mathfrak{p}}

\newcommand{\1}{{\bf 1}}

\newcommand{\eof}[1]{e({#1})}
\newcommand{\FT}[1]{\widehat{#1}}
\newcommand{\vof}[1]{{\boldsymbol #1}}
\newcommand{\trace}[2]{{\Tr_{#1}(#2)}}

\def\dd{{\,{\rm d}}}

\def\translate{{{\rm T}}}
\def\Hconstant{{{\rm H}}}
\def\S{{{\rm S}}}

\DeclareMathOperator{\Tr}{Tr}

\newtheorem{theorem}{Theorem}%[section]
\newtheorem{lemma}[theorem]{Lemma}
\newtheorem{proposition}[theorem]{Proposition}

\theoremstyle{remark}
\newtheorem{remark}{Remark}%[theorem]
\newtheorem*{prosexei}{Warning}%[theorem]
\title[Reinforcing a Philosophy: The moment curve]{\vspace*{-2.0cm}
Reinforcing a Philosophy: 
Littlewood--Paley theory for the moment curve over generic local fields
}
\author[K. Hughes]{Kevin Hughes}
\address{
School of Computing, Engineering \& the Built Environment, 
Edinburgh Napier University, 
Edinburgh EH10 5DT, UK
}
\email{khughes.math@gmail.com} 
\begin{document}

\begin{abstract}
Using the Girard--Newton formulae, I give a simple proof of isotropic square function estimates for extension operators along the moment curve in generic local fields. 
Using Bezout's Theorem and the Implicit Function Theorem, I give an alternate, sharper proof for real or complex non-degenerate, polynomial curves. 
\end{abstract}

\maketitle

% \vspace*{-0.5cm}
% % % 
\section{Introduction}
% % % 

The purpose of this paper is to expose two simple arguments for square function estimates of the moment curve by utilizing the number theory paradigm: 
\[ \texttt{\color{purple} Special subvarieties dictate the analysis}. \]
This paradigm will be investigated in the context of proving a square function estimate for the moment curve over generic local fields. 
My hope is to provide some context to recent works in harmonic analysis and relate them to  common, elementary ideas in analytic number theory. % overlooked classical works
Stating the results requires a brief set-up taken from \cite{BBH}.

% % 
%\subsection{Set-up and the main result}
% % 

In this paper, let $\lf$ denote a (one-dimensional) local field; its non-trivial topology is determined by the metric associated to a fixed absolute value $|\cdot|_\lf : \lf \to \R_{\geq0}$. 
On $\R$, this will be the usual absolute value. On $\C$, we choose our absolute value so that $|x+iy|_\C = \max\{|x|,|y|\}$ where the latter is the usual absolute value on $\R$. 
In the results below, $\lf$ will be fixed; so, I usually suppress the dependence on $K$ in future notations. For $n \in \N$, extend the metric to $\lf^n$ by defining $|\vof{x}| = \max\{|x_1|,\dots,|x_n|\}$ for $\vof{x}=(x_1,\dots,x_n) \in \lf^n$. 
Always assume that the dimension $n$ is at least two.

When $\lf$ is a non-Archimedean local field, let $\roi := \{ x : |x| \leq 1 \}$ denote its ring of integers and $\oddprime := \{ x : |x|<1 \}$ denote its maximal ideal. 
%The residue field of $\lf$, defined as the quotient $\roi/\oddprime$, is a finite field. 
%Define the degree of the extension $\lf/\rlf$ as $[\lf:\rlf]$. 
The image of $|\cdot|_{\lf}$ is isomorphic to $\Z$ as an abelian group and $\lf$ comes with a uniformizing element, say $\uniformizer$, generating this group. 
%We assume that the metric on the non-Archimedean local field $\lf$ is normalized such that $|\uniformizer|=\#(\roi/\oddprime)^{-1/[\lf:\rlf]}$. 
When $\lf=\R$, let $\roi=[0,1]$, and when $\lf=\C$, let $\roi:=\{x+y\sqrt{-1} : x,y \in [0,1]\}$. 

For a non-Archimedean local field $\lf$, define its scales as $\scales(\lf) := \{ |\uniformizer|^{s} : s \in \Z_{\geq0} \}$. For a scale $\delta \in \scales(\lf)$, let $\partition_{\delta}(\lf)$ denote the partition of $\roi$ into closed (and also open) balls of radius $\delta$; that is, $\partition_{\delta} := \{ (i+\oddprime^s)_{i \in \roi/\oddprime^s} \}$ for some $s \in \N$. %$\partition_{\delta}(\lf)$ is finite for each $\delta>0$. 
Over $\R$, define its scales as $\scales(\R) := \{ R^{-1} : R \in \N \}$, and for each $\delta$ in $\scales(\R)$, define the partition 
\(
\partition_{\delta}(\R) 
:= \{ [j\delta,(j+1)\delta) : i = 0,\dots,\delta^{-1}-1 \}
.\) 
Over $\C$, define its scales as $\scales(\C) := \{ R^{-1} : R \in \N \}$, and for each $\delta$ in $\scales(\C)$, define the partition 
\(
\partition_{\delta}(\C) 
:= \{ [j\delta,(j+1)\delta) + i[k\delta,(k+1)\delta) :  j,k=0, \dots, \delta^{-1}-1 \}
.\) 

For non-Archimedean fields, we fix a non-trivial additive character $\eof{\cdot}$ such that $\eof{\roi}=1$ and $\eof{1/\uniformizer} \neq 1$. 
On $\R$, $\eof{t} := e^{-2\pi it}$ denotes the usual character, and on $\C$, $\eof{z} := e^{-2\pi i \trace{\C/\R}{z}/2} = e^{-2\pi i {x}}$ for $z=x+iy \in \C$ with $x,y \in \R$ and $i:=\sqrt{-1}$. (This is the usual character when interpreting $\C$ as $\R^2$ in the usual way. And over $\R$ and $\C$, $\pi$ denotes the standard mathematical constant.) 
On each local field, the Haar measure $\dd{\xi}$ on $\lf$ is normalized so that the measure of $\roi$ is 1. 
For a function $f$, write its Fourier transform as $\FT{f}$. 

%Let $\curve(T)=\big(\curve_1(T),\dots,\curve_n(T)\big) \in (\lf[T])^n$ be a polynomial curve. 
Throughout, assume that the functions $f$ are compactly supported, measurable functions. For a fixed local field $\lf$, a measurable set $I$ in $\lf$ and a curve $\curve : I \to \lf^n$, define the \emph{extension operator along $\curve$ over $I$} as  
\[
E_{I}f(\vof{x}) 
:= 
\int_{I} f(\xi) \eof{\curve(\xi) \cdot \vof{x}} \dd{\xi} 
\quad \text{for points} \quad \vof{x} \in \lf^n.
\]
For $\delta \in \scales(\lf)$, define the \emph{square function along $\curve$ at scale $\delta$} as 
\[
S_{\delta} f(\vof{x}) 
:= 
\Big( \sum_{J \in \partition_{\delta}} |E_{J}f(\vof{x})|^2 \Big)^{1/2} 
%{}_.
.\]
Suppose that $w : \lf^n \to \lf$ is a reasonable function. For each $\vof{c} \in \lf^n$ and each $R \in \lf \setminus \{0\}$, define 
\[ 
w_{\vof{c},|R|}(\vof{x}) 
:= w\big(\frac{\vof{x}-\vof{c}}{|R|}\big) 
\quad \text{for} \quad \vof{x} \in \lf^n 
.\] 
I will also write this as $w_B$ when $B$ is the associated box with center $\vof{c}$ and sidelengths $|R|$. 
Finally, define the constant 
\begin{equation*}\label{inequality:definition}
\Hconstant_\curve(w) 
:= 
\sup_{\delta \in \scales(\lf)}\sup_{\vof{c} \in \lf^n} \sup_{f \neq 0}  
\| E_{\roi} f \|_{L^{2n}(w_{\vof{c},\delta^{-1}})} \big/ \| S_{\delta} f \|_{L^{2n}(w_{\vof{c},\delta^{-1}})} 
.\end{equation*}

Suppose that $w:\R \to \R$ is a Schwartz function which is non-negative, at least 1 on the unit interval $[0,1]$ and for which $\widehat{w}$ is supported on $[-1,1]$ and non-negative. 
We define the $n$-dimensional version of $w$ as $W(\vof{x}) := w(x_1)\cdots w(x_n)$ for $\vof{x} \in \R^n$. Then $W$ is at least 1 on the cube $[0,1]^n$ while $\widehat{W}$ is supported on the cube $[-1,1]^n$. 
On $\C$, define our function $W(\vof{x}+i\vof{y})$ to be $W(\vof{x})W(\vof{y})$ for $\vof{x}, \vof{y} \in \R^n$. 
Finally, on a non-Archimedean field $\lf$, we define $W := \1_{\{\vof{z} \in \lf^n : |\vof{z}| \leq 1\}}$. Then $W$ is supported on and is 1 on the cube $\roi^n$ while $\widehat{W}$ is supported, and also 1, on the cube $\roi^n$. 

For a non-Archimedean local field $\lf$, define the constants $C_{\lf} := 1$. For $\lf=\R$, define the constant $C_{\R} := 7$. For $\lf=\C$, define the constant $C_{\C} := 7^{2}$. 
we can finally state the theorems. 
\begin{theorem}\label{theorem:main}
Let $\lf$ be a local field and $\curve(T) := (T,T^2,\dots,T^n)$. Assume that the characteristic of $K$ is either 0 or greater than $n$. 
We have the inequality $\Hconstant_\curve(W) \leq (C_{\lf} n)^{1/2}$. 
In other words, for each scale $\delta \in \scales(\lf)$ and $\vof{c}$ in $\lf^n$, we have the inequality 
\begin{equation}\label{ineq:main}
\| E_{\roi} f \|_{L^{2n}(W_{\vof{c},\delta^{-n}})} 
\leq (C_{\lf} \cdot n)^{1/2} \; 
\| S_{\delta} f \|_{L^{2n}(W_{\vof{c},\delta^{-n}})} 
.\end{equation}
\end{theorem}
\noindent The bound \eqref{ineq:main} is close to sharp. 
Indeed, Stirling's Approximation $n! \sim \sqrt{2\pi n} (n/e)^{n}$ and the following lower bound implies that \eqref{ineq:main} is asymptotically off by a constant. 
\begin{theorem}\label{theorem:lower_bound}
Fix $n \geq 2$. 
For each local field $\lf$ and polynomial curve $\curve \in \lf[T]^n$, 
\begin{equation}
\Hconstant_{\curve}(W) \geq (n!)^{1/2n}
.\end{equation}
\end{theorem}
% I do not need any condition on the characteristic of the local field here :) 

The problem of finding sharp constants and extremizers for square function estimates appears interesting. 
Inspired by \cite{Wooley:symmetric, PW}, I give two sharper, more general versions of Theorem~\ref{theorem:main} for Archimedean fields using calculus and classical algebraic geometry. 
Define $\eta_{\lf}$ to be 1 if $\lf$ is $\R$ and 2 if $\lf$ is $\C$. 
For a curve $\curve : \roi \to \lf^n$, define its Lipschitz norm 
\[
\ell(\curve) = \sup_{i=1,\dots,n} \sup_{s,t \in \roi} \frac{|\curve_i(t)-\curve_i(s)|}{|t-s|}
.\]
Observe that $\ell(\curve)$ is finite when $\curve$ is polynomial and $\roi$ is compact. 
%Assume that the dimension $n$ is at least two for the remainder of this section. 

\begin{theorem}\label{theorem:Bezout}
Let $\lf$ be $\R$ or $\C$ and $\curve \in \lf[T]^n$ be a non-degenerate polynomial curve. 
We have the inequality $\Hconstant_\curve(W) \leq \big( 2\lceil \ell(\curve) \rceil +1 \big)^{\eta_{\lf}/2} \big( \prod_{i=1}^{n} \deg(\curve_i) \big)^{1/2n}$. 
\end{theorem}

\noindent When $\gamma$ is the moment curve, the proof of Theorem~\ref{theorem:Bezout} demonstrates that we have `strong diagonal behavior'. For more general non-degenerate curves, any off-diagonal behavior is constrained by Bezout's theorem. 
These give the special subvarieties under-pinning square function estimates. 

Over $\R$, there are much sharper bounds for fewnomials. 
\begin{theorem}\label{theorem:fewnomial}
Let $\curve$ be a non-degenerate, polynomial curve in $\R^n$ such that the total number of monomials appearing in $\curve$ is $M$. 
We have the inequality $$\Hconstant_\curve(W) \leq \big( 2\lceil \ell(\curve) \rceil +1 \big)^{1/2} \big( 2^{M(M-1)/2} (n+1)^{M} \big)^{1/2n}.$$ 
\end{theorem}

\noindent With only a minor modification, the proofs of Theorem~\ref{theorem:Bezout} and \ref{theorem:fewnomial} generalize to non-degenerate Pfaff curves in $\R^n$. The statement is more technical than its proof, so I refer the reader to \cite{Khovanskii:transcendental, Khovanskii:fewnomials} to extract what is needed. 

% \begin{remark}
% It is easy to replace $L^{2n}$ above with $L^{2k}$ for $1 \leq k \leq m$. 
% It seems like there might be a nice abstract argument to this. 
% \end{remark}

% % 
%\subsection{Motivation}
% % 

In the form described above, \cite{GGPRY} proved Theorem~\ref{theorem:main} for \emph{real} non-degenerate curves; their result is far more general at the cost of an inexplicit constant. A new proof for Theorem~\ref{theorem:main} when $n=2$ was given in \cite{BBH}. 
% Around that time, I proved Theorem~\ref{theorem:main}. 
Subsequently, \cite{WH} proved the bound $\Hconstant_\curve(W)^{2n} \leq n!$ for non-Archimedean fields of sufficiently large characteristic. 
% After I corrected a mistake on their part which initially claimed the equality immediately below. 
Combined with Theorem~\ref{theorem:lower_bound}, we see that $\Hconstant_\curve(W) = (n!)^{1/2n}$ for the moment curve over non-Archimedean local fields of characteristic 0 or characteristic exceeding $n$. % the degree of the moment curve.

The arguments proving Theorems~\ref{theorem:main} and \ref{theorem:lower_bound} are closely related to the starting point for Vinogradov's celebrated mean value theorems. While these arguments should be well-known, they seem to be known to only a few. 
% I indicated the proof of Theorem~\ref{theorem:main} herein to {WH} prior to their announcement of \cite{WH}, but they chose to ignore it. This paper gives the details. 
% Two (other) people communicated to me that they had derived a similar argument albeit only over $\R$. 
After consulting with experts, I believe the proofs of Theorems~\ref{theorem:Bezout} and \ref{theorem:fewnomial} appear to be new. 
I now give a concise overview of some antecedents to these results. 

Historically, the main interest in such square function estimates arises from exponential sum estimates in analytic number theory; specifically, Vinogradov's mean value theorems and Waring's problem where the main choice of functions for input into \eqref{ineq:main} were (after a rescaling) $f:=\sum_{i=1}^N \delta_{i/N}$ as $N$ tends to infinity. For these functions, the first result is due Vinogradov; see \cite[Vinogradov's `pigeon-hole lemma' on page 12]{KV} and \cite[Chapter~4, Section~2]{Hua1965}. 
Linnik subsequently developed a version of Vinogradov's methods for local fields; see \cite[pages 71-72]{Montgomery:10lectures}. 
%{The transversality assumption is that a sequence of $n$ congruences $\pmod{p^n}$ are all distinct modulo $p$. With this assumption one may quickly deduce strong diagonal behavior. Crucially in our argument, we do not have this tranversality at our disposal.} 
% {The same methods give $\ell^2 \to L^{2n}$-discrete restriction estimates for the degree $n$ moment curve; this is an unpublished result of mine going back to 2013.} 

In harmonic analysis, related results appeared later. 
A Cantor--Lebesgue theorem for the circle appeared in \cite{Cooke}, followed by a $\ell^2 \to L^4$-discrete restriction estimate for the circle in \cite{Zygmund:2variables}. 
For more general functions, \cite{GGPRY} attributes a version of Theorem~\ref{theorem:main} to \cite{F} when $n=2$ and $\lf=\R$ while \cite{WH} attributes a version of this theorem to \cite{P1, P2} when $n=3$ and $n \geq 4$ respectively; each version is for square functions related to Bochner--Riesz operators instead of the square functions herein. 

Uniting all of these works is the aforementioned paradigm from number theory. %: \texttt{special subvarieties dictate the analysis}. 
My use of `analysis' there does not refer to any mathematical field, but to quantitative considerations of problems related to certain equations. This paradigm has long played a salient role in number theory and is prominent in analytic number theory through Manin's conjecture and the circle method. 
In harmonic analysis, this paradigm is understated despite its long-standing presence therein. 

Let me take a moment to describe how this paradigm arises in the aforementioned works. 
In Vinogradov's pigeon-hole lemma and \cite[`Linnik's Lemma' on pages 71-72]{Montgomery:10lectures}, this paradigm arises by demonstrating diagonal behavior for certain Vinogradov systems with an added transverality assumption; over $\R$ transversality can be encoded as distinct 1-separated integers (note that Vinogradov's methods easily generalize to 1-separated real points), and over the $p$-adics, transversality is usually encoded $1/p$ separation between the $p$-adic intervals. 
In \cite{GGPRY, BBH, WH} and this paper, this paradigm appears in the diagonal behavior of isotropic boxes covering a neighborhood of the underlying variety. 
In \cite{Cooke, Zygmund:2variables} this paradigm appears in the geometry of two intersecting circles. 
In \cite{F, P1, P2}, this paradigm appears in the almost orthogonality of off-diagonal Kakeya type information. 
In all cases, this information is drawn out through a reduction using orthogonality and multilinearity considerations.

% \bigskip 
% Such results for square functions came to my attention in \cite{GGPRY} where the authors therein attribute a version (for Bochner--Riesz operators) of Theorem~\ref{theorem:main} to \cite{F} when $n=2$ and $\lf=\R$. The proof in \cite{F} proceeded by analyzing Kakeya information (the intersection of certain families of rectangles) and $L^4$ bilinear arguments. The $L^4$ bilinear arguments of \cite{F} relate to those of \cite{Cooke} where a discrete restriction version of Theorem~\ref{theorem:main} could be extracted as in \cite{Zygmund:2variables}. 

% There is a key difference between \cite{WH} and this paper: while the proof herein is much simpler, it comes at the loss of some structural information which yields a worse constant in the square function estimate below. In short, the key difference is that I sacrifice showing a strong diagonal behavior property as in \cite{GGPRY, WH} for a weaker, yet sufficient property. 
% The strong diagonal behavior demonstrated in \cite{GGPRY,WH} is a beautiful generalization of the classical Diophantine result in Vinogradov's mean value theorems. 

% % 
\subsection*{Outline of the paper}
% % 

In Section~\ref{section:reduction}, I reduce the proof of Theorem~\ref{theorem:main} to bounds on the number of special subvarieties via a standard use of orthogonality and the Cauchy--Schwarz inequality. This reduction is encoded in Lemma~\ref{lemma:an_uncertain_CS}. 
In Section~\ref{section:main_prop}, I prove Proposition~\ref{prop:main} which, in combination with Lemma~\ref{lemma:an_uncertain_CS}, yields Theorem~\ref{theorem:main}. At the end of this section, I prove the lower bound Theorem~\ref{theorem:lower_bound}. 
In Section~\ref{section:variants}, I prove Theorems~\ref{theorem:Bezout} and \ref{theorem:fewnomial}. 
In Section~\ref{section:SO}, I discuss the connection to Superorthogonality.

% % 
\subsection*{Acknowledgements}
% % 

I thank J. de Dios Pont, A. Mudgal, O. Robert and T. Wooley for their feedback. %on an earlier draft. 

This work was supported by the Additional Funding Programme for Mathematical Sciences, delivered by EPSRC (EP/V521917/1) and the Heilbronn Institute for Mathematical Research.

% % 
%\subsection*{Funding}
% % 

%This work was supported by the Additional Funding Programme for Mathematical Sciences, delivered by EPSRC (EP/V521917/1) and the Heilbronn Institute for Mathematical Research. 

\section{Scratch}
{
}

%\newpage 
% % % 
\section{Reduction to special subvarieties}\label{section:reduction}
% % % 

%warning{`The Uncertain Cauchy--Schwarz argument' instead of Cordoba--Fefferman? Or anything a bit more descriptive.} 

Fix a local field $\lf$ and a curve $\curve : \roi \to \lf^n$. 
For each scale $\delta \in \scales(\lf)$, each $\epsilon>0$ and each $n$-tuple $\vof{I}$ in $\partition_{\delta}^n$, define $\syzygies(\delta,\vof{I}; \epsilon)$ to be the set of $n$-tuples $\vof{J} \in \partition_{\delta}^n$ such that 
\begin{equation}\label{close_points}
| \sum_{i=1}^n \big( \curve(t_i)-\curve(s_i) \big) | 
\leq \epsilon
\end{equation} 
for some $\vof{s} \in \vof{I}$ and some $\vof{t} \in \vof{J}$. 
In other words, 
\begin{equation*}%\label{close_points}
\syzygies(\delta,\vof{I}; \epsilon)
:= 
\big\{ \vof{J} \in \partition_{\delta}^n : \text{ there exist points } \vof{s} \in \vof{I}, \vof{t} \in \vof{J} \text{ satisfying } \eqref{close_points} \big\}
.\end{equation*}
Define $\S_{\curve} := \max_{\delta \in \scales} \max_{\vof{I} \in \partition_\delta^n} |\syzygies(\delta,\vof{I}; \delta^{n})|$.

\begin{lemma}\label{lemma:an_uncertain_CS}
Let $\lf$ be a local field, $n \geq 2$ and $\curve \in \lf[T]^n$ be a curve in $\lf^n$. 
Then 
\begin{equation}
\Hconstant_{\curve}(W) \leq (\S_{\curve})^{1/2n} 
.\end{equation}
\end{lemma}
% No assumption on the characteristic is needed here. 

\begin{proof}
Fix $n \geq 2$. 
Fix a local field $\lf$.  
Fix a scale $\delta \in \scales(\lf)$. 
Let $B$ be a box of sidelengths $\delta^{-n}$ in $K^n$. Without loss of generality, it suffices to take $B$ centered at the origin. 
Write $E_{\roi} = \sum_{I \in \partition_{\delta}} E_{I}$ by the linearity of integration. Fubini's theorem implies that  
\begin{align*}
\| E_{\roi} f \|_{L^{2n}(W_B)}^{2n} 
&= 
\sum_{\vof{I} \in \partition_{\delta}^n} \sum_{\vof{J} \in \partition_{\delta}^n} \int_{\lf^n} \big( \prod_{i=1}^n E_{I_i}f(\vof{x}) \big) \big( \prod_{j=1}^n \overline{E_{J_j}f(\vof{x})} \big) W_B(\vof{x}) \dd{\vof{x}}
.\end{align*}
Fourier inversion and the properties of $W_B$ imply that 
\begin{align*}
\| E_{\roi} f \|_{L^{2n}(W_B)}^{2n} 
&= 
\sum_{\vof{I} \in \partition_{\delta}^n} \sum_{\vof{J} \in \syzygies(\delta,\vof{I};\delta^{n})} \int \big( \prod_{i=1}^n E_{I_i}f \big) \big( \prod_{j=1}^n \overline{E_{J_j}f} \big) W_B
.\end{align*}
For more details of this orthogonality, see \cite[Section~6]{GGPRY}. 
% or the proof of Lemma~4.1 in \cite{BBH} 

Applying Fubini's theorem once more as well as the Cauchy--Schwarz inequality on the sum over $\vof{I} \in \partition_{\delta}^n$, we deduce that 
\begin{align*}
\| E_{\roi} f \|_{L^{2n}(W_B)}^{2n} 
&= 
\int \sum_{\vof{I} \in \partition_{\delta}^n} \big( \prod_{i=1}^n E_{I_i}f \big) \big( \sum_{\vof{J} \in \syzygies(\delta,\vof{I};\delta^{n})} \prod_{j=1}^n \overline{E_{J_j}f} \big) W_B
\\&\leq 
\int \Big( \sum_{\vof{I} \in \partition_{\delta}^n} |\prod_{i=1}^n E_{I_i}f|^2 \Big)^{1/2} \Big( \sum_{\vof{I} \in \partition_{\delta}^n} |\sum_{\vof{J} \in \syzygies(\delta,\vof{I};\delta^{n})} \prod_{j=1}^n \overline{E_{J_j}f}|^{2} \Big)^{1/2} W_B
%\\&= 
%\int \Big( S_{\delta}f \Big)^n \Big( \sum_{\vof{I} \in \partition_{\delta}^n} \big|\sum_{\vof{J} \in \syzygies(\delta,\vof{I};\delta^{n})} \prod_{j=1}^n \overline{E_{J_j}f}\big|^{2} \Big)^{1/2} W_B
.\end{align*}
Applying the Cauchy--Schwarz inequality on the inner sum over $\vof{J} \in \syzygies(\delta,\vof{I};\delta^{n})$ and the triangle inequality on the second outer sum over $\vof{I} \in \partition_{\delta}^n$, we deduce that 
\begin{align*}
\| E_{\roi} f \|_{L^{2n}(W_B)}^{2n} 
&\leq 
\int \Big( S_{\delta}f \Big)^n \Big( \sum_{\vof{I} \in \partition_{\delta}^n} |\syzygies(\delta,\vof{I};\delta^{n})| \sum_{\vof{J} \in \syzygies(\delta,\vof{I};\delta^{n})} |\prod_{j=1}^n {E_{J_j}f}|^{2} \Big)^{1/2} W_B
\\&\leq 
\S_{\curve}^{1/2} \int \Big( S_{\delta}f \Big)^n \Big( \sum_{\vof{I} \in \partition_{\delta}^n} \sum_{\vof{J} \in \syzygies(\delta,\vof{I};\delta^{n})} |\prod_{j=1}^n {E_{J_j}f}|^{2} \Big)^{1/2} W_B
.\end{align*}

By the symmetry of the inequalities \eqref{close_points}: $\vof{J} \in \syzygies(\delta,\vof{I};\delta^{n})$ if and only if $\vof{I} \in \syzygies(\delta,\vof{J};\delta^{n})$. 
Use this symmetry to invert the double sum and apply the triangle inequality on the innermost summand over $\vof{I} \in \syzygies(\delta,\vof{J};\delta^{n})$ to deduce that 
\begin{align*}
\| E_{\roi} f \|_{L^{2n}(W_B)}^{2n} 
&\leq 
\S_{\curve}^{1/2} \int \Big( S_{\delta}f \Big)^n \Big( \sum_{\vof{J} \in \partition_{\delta}^n} \sum_{\vof{I} \in \syzygies(\delta,\vof{J};\delta^{n})} |\prod_{j=1}^n {E_{J_j}f}|^{2} \Big)^{1/2} W_B
\\&\leq 
\S_{\curve} \int \Big( S_{\delta}f \Big)^n \Big( \sum_{\vof{J} \in \partition_{\delta}^n} |\prod_{j=1}^n {E_{J_j}f}|^{2} \Big)^{1/2} W_B
%\\&= 
%\S_{\curve} \int \Big( S_{\delta}f \Big)^{2n} W_B 
= 
\S_{\curve} \|S_{\delta}f\|_{L^{2n}(W_B)}^{2n} 
.\end{align*}
Taking $2n^{\textrm{th}}$-roots on both sides of the inequality completes the proof.
\end{proof}

%\newpage 
% % % 
\section{Diagonal behavior of the Vinogradov system}\label{section:main_prop}
% % % 

Our main estimate uniformly bounds the size of $\syzygies(\delta,\vof{I}; \delta^{n})$ which implies that $\S_{\curve}$ is finite. 
Via Lemma~\ref{lemma:an_uncertain_CS}, our particular bound for the cardinality of $\syzygies(\delta,\vof{I}; \delta^{n})$ immediately implies Theorem~\ref{theorem:main}; I leave this implication to the reader. 
\begin{proposition}\label{prop:main}
Let $\curve := (T,T^2,\dots,T^n) $ be the moment curve in a fixed local field $\lf$ of characteristic 0 or greater than $n$. 
We have the bound 
\begin{equation}\label{bound:syzygies}
\S_{\curve} \leq (C_{\lf} \cdot n)^n
.\end{equation}
\end{proposition}

%\noindent 
See \cite[Proposition~1.3]{GGPRY}, \cite[Proposition~3.1]{BBH} and \cite[Proposition~1.2]{WH} for analogous propositions. 
Compare with classical versions of this proposition such as its original form due to Vinogradov exposed in \cite[Vinogradov's `pigeon-hole lemma' on page 12]{KV}, \cite[pages 37--44]{Vin:book}, \cite[Chapter~4, Section~2]{Hua1965} and \cite[Lemma~6.3 in Chapter~VI on pages 121-122]{Titchmarsh}. 
Linnik developed a $p$-adic version exposed in \cite[pages 58-60]{Vaughan:book} and \cite[pages 71-72]{Montgomery:10lectures}; these use an added transversality assumption between the intervals which we crucially do not have at our disposal.

Let me take a moment to describe the idea underlying Proposition~\ref{prop:main}. 
Suppose that $\vof{s}$ and $\vof{t}$ are points in $\lf^n$. 
If $\vof{t}$ is a permutation of $\vof{s}$, then this pair of points is a solution to \eqref{close_points} with $\epsilon=0$. 
It transpires that the converse is true. 
To see this converse, define the elementary symmetric polynomials $\sigma_j(X_1,\dots,X_n)$ to be $\sum_{i_1<i_2<\cdots<i_j} X_{i_1} \cdots X_{i_j}$ for $j\in\N$ and recall the Girard--Newton equations: 
\begin{equation}\label{identity:GN}
(-1)^{j-1}j \sigma_j(X_1,\dots,X_n) = \sum_{i=0}^{j-1} (-1)^i \big( X_1^{j-i}+\cdots+X_n^{j-i} \big) \sigma_i(X_1,\dots,X_n) 
.\end{equation} 
The Girard--Newton equations imply that for each coordinate $t_j$, there exists an $s_i=t_j$. 
The special subvarieties discussed previously are these `almost diagonal' varieties given by $t_j=s_i$ for each $j$ and some $i$. 
Fixing a point $\vof{s}$, the number of such points $\vof{t}$ (and equivalently special subvarieties) is at most $n^n$. 
The critical insight underlying Proposition~\ref{prop:main} is that a fattened version of this argument continues to hold for $\syzygies(\delta,\vof{I};\delta^{n})$. 
% An unobvious abstraction of this critical insight leads to a new, strong type of superorthogonality; see `Type IV superorthogonality' in \cite{newSO}. 

Going further, the Girard--Newton equations \eqref{identity:GN} and the Fundamental Theorem of Algebra imply that permutations are the only solutions. Obtaining a fattened version of this `strong diagonal property' was an essential feature in \cite{WH}; the strong diagonal property took the form: if $\vof{J} \in \syzygies(\delta,\vof{I},\delta^n)$, then $\vof{J}$ is a permutation of $\vof{I}$. 
%Upon doing so, one obtains the bound of $n!$ in place of $n^n$ for non-Archimedean fields. 
% However, this strong diagonal property is not necessary for Theorem~\ref{theorem:main}. Indeed Proposition~\ref{prop:main} and its proof give a slightly weaker diagonal behavior. 
We will see this strong diagonal behavior again in Proposition~\ref{prop:Bezout} below. 

\begin{proof}[Proof of Proposition~\ref{prop:main}]
Fix $n \geq 2$. 
First, I will prove the proposition for non-archimedean local fields of characteristic 0 or of characteristic greater than $n$. For these cases, the constant $C_{\lf}$ appearing in statement of the proposition is 1. 
At the end of the proof, I will indicate the necessary changes when $\lf$ is $\R$ or $\C$. 

Fix a non-archimedean local field $\lf$ of characteristic 0 or of characteristic greater than $n$. 
Fix a scale $\delta \in \scales(\lf)$ and an $n$-tuple of intervals $\vof{I}$ in $\partition_{\delta}^n$. 
It suffices to show the bound 
%\begin{equation*}%\label{bound:syzygies}
\(|\syzygies(\delta,\vof{I}; \delta^{n})| 
\leq 
n^n\). 
%.\end{equation*}
For $\vof{s} \in \lf^n$, define the univariate polynomial 
\[
G(\vof{s};X) := \prod_{i=1}^n (X-s_i)
\] 
where $X$ is the variable. 
If $\vof{s} \in \vof{I}$, $\vof{J} \in \syzygies(\delta,\vof{I}; \delta^{n})$ and $\vof{t} \in \vof{J}$, then the Girard--Newton equations \eqref{identity:GN} imply that \eqref{close_points} holds, for $\epsilon=\delta^n$, with the elementary symmetric polynomials in place of the power symmetric polynomials. In turn, this implies 
\begin{equation}\label{inequality:GN}
|G(\vof{t};x) - G(\vof{s};x)| \leq \delta^{n} 
\text{ for all } x \in \roi
.\end{equation}
Taking $x = t_j$, we find that 
%\begin{equation}\label{small_product}
\(|G(\vof{s};t_j)| \leq \delta^{n}\) 
%\end{equation}
for each $j=1,\dots,n$. 
The pigeonhole principle implies that there exists an $i=1,\dots,n$ such that $|s_i-t_j| \leq \delta$. 
Since the local field $\lf$ is non-archimedean, $t_j \in I_i$ where $\vof{I}$ is written as $(I_1,\dots,I_n)$, and as a result, $J_j = I_i$. 
It is now obvious that there are at most $n$ choices for each of the $n$ coordinates. Therefore, there are at most $n^n$ choices overall.

Now assume that $\lf$ is $\R$ or $\C$. The first difference in the proof is that the transfer from power symmetric polynomials to elementary symmetric polynomials induces a loss of a factor at most $2n^2$ in \eqref{inequality:GN}. 
To see this, observe that \eqref{identity:GN} implies%\footnote{One might think that this constant should be $2\sqrt{2}$ when $\lf=\C$, but I am using the $\ell^\infty$-metric on the real and imaginary parts.} 
\[
|j\sigma_j(s_1,\dots,s_n) - j\sigma_j(t_1,\dots,t_n)| 
\leq 2n\sum_{i=0}^{j-1} \big| \sigma_i(s_1,\dots,s_n)-\sigma_i(t_1,\dots,t_n) \big| 
.\]
This implies that for each $j=1,\dots,n$, we have the bound
\[
|\sigma_j(s_1,\dots,s_n) - \sigma_j(t_1,\dots,t_n)| 
\leq 2n \delta^n
.\] 
Consequently, for all $n \geq 2$, we have 
\begin{equation*}%\label{inequality:GN}
|G(\vof{t};x) - G(\vof{s};x)| 
\leq 2n^2 \delta^n
\leq (3\delta)^{n} 
\text{ for all } x \in \roi
.\end{equation*}
The second difference arises from the necessity to account for possible neighbors of intervals arising from the above inequality. This loses a factor of $(3\cdot2+1)^n=7^n$ or $(3\cdot2+1)^{2n}=49^n$ in the estimate for $|\syzygies(\delta,\vof{I};\delta^{n})|$ for $\R$ and $\C$ respectively. %See \cite{BBH} for similar considerations. 
The remaining details of the proof for Archimedean fields are left to the reader. 
\end{proof}

% \begin{remark}
%The combinatorics applied above is quite crude. 
% There are several improvements to \eqref{bound:syzygies} using some simple combinatorics. Such combinatorics easily decrease $n^n$ to $\max_{m=1,\dots,n} \big\{ n(n-1)\cdots(n-m+1) \cdot m^{n-m} \big\}$, or to Stirling numbers of the second kind. 
% When $n=2$ or 3, the combinatorics is simple enough so that one easily refines the bounds to $2!$ and $3!$ respectively. 
% Unfortunately, the best bound deducible from this method exceeds $n!$ when $n \geq 4$. 
% 
% For non-Archimedean fields, the bound $\Hconstant_\curve(W)^{2n} \leq n!$ was obtained in \cite{WH}. {By Stirling's Approximation $n! \sim \sqrt{2\pi n} (n/e)^{n}$, we find that \eqref{ineq:main} is asymptotically off by a factor of Euler's number $e$.} 
% Theorem~\ref{theorem:Bezout}, proved in Section~\ref{section:variants} below, yields $\Hconstant_\curve(W)^{2n} \leq 5^{\eta_\lf n}n!$ for $\lf=\R$ or $\C$. 
%Since determining sharp bounds almost always requires a deep understanding of the symmetries underlying the problem, this lower bound indicates that the sharp bound for $\Hconstant_\curve(W)$ requires a deeper argument than the one given above. 
%Such an argument was given in \cite{WH}, and an alternate argument (for Archimedean local fields) is given in Section~\ref{section:variants}. 
% \end{remark}

% \begin{remark}
There is a small analytic improvement for Archimedean fields. 
For $n \geq 7$, we can improve $2n^2 < 3^n$ to $2n^2 < 2^n$. Therefore, when $n \geq 7$, we have a factor of $(2\cdot2+1)^n=5^n$ or $(2\cdot2+1)^{2n}=25^n$ in the above estimate for $|\syzygies(\delta,\vof{I};\delta^{n})|$ for $\R$ and $\C$ respectively. 
For any $\rho \in (1,2)$, we have $2n^2 < \rho^n$ for sufficiently large $n$ and similar improvements can be made. Asymptotically as the dimension tends to infinity, this improves the factors $C_\R$ towards $3$ and $C_\C$ towards $3^{2}$. 
% \end{remark}

\begin{prosexei}\label{remark:mistake}
% Their "typos":
A common mistake is the following: For an $n$-tuple of intervals $\vof{I} \in \partition_{\delta}^n$, there are exactly $n!$ permutations of $\vof{I}$. This is not true because there are fewer than $n!$ permutations when the same interval may appears more than once in the $n$-tuple. The effect of this is that \emph{one does not deduce} for each function $f$ the equality
\begin{equation*}%\label{ineq:main}
\| E_{\roi} f \|_{L^{2n}(W_B)} 
%\leq (C_{\lf,n}H_n)^{2/n}
= (n!)^{1/(2n)}
\| S_{\delta} f \|_{L^{2n}(W_B)} 
\end{equation*}
in non-Archimedean fields; instead, only the upper bound (that is, $\leq$) is deduced. 
\end{prosexei}

Despite this, that upper bound is sharp. This brings us to the proof of Theorem~\ref{theorem:lower_bound}. 
\begin{proof}[Proof of Theorem~\ref{theorem:lower_bound}]
By mollifying the functions $f := \sum_{i=1}^{N} \delta_{i/N}$ for appropriate, large $N \in \N$, a lower bound for $\Hconstant_{\curve}$ relates to counting solutions to the system of equations $\sum_{i=1}^n \big( \curve(t_i)-\curve(s_i) \big) = 0$ where $s_1, t_1, \dots, s_n, t_n \in \{1/N, 2/N, \dots, N/N\}$. To be precise, $\| E_{\roi} f \|_{L^{2n}(\translate_{\vof{c}}W_{\delta^{-n}})}^{2n}$ counts the number of such solutions which is $n!N^n +O(N^{n-1})$. One see this by fixing one set of variables, say $\vof{s}$, and observing that any permutation $\vof{t}$ of $\vof{s}$ is also a solution. Additionally, there are strictly less than $n! N^n$ solutions by the reasoning in Remark~\ref{remark:mistake}. Meanwhile, $\| S_{N^{-1}} f \|_{L^{2n}(\translate_\vof{c}w)}^{2n}$ is simply the diagonal contribution $N^n$. (I have omitted a factor of $\int W_N$ which appears in calculating both $L^{2n}$-norms.) 
Taking the limit as $N$ goes to infinity, we see that $\Hconstant_{\curve} \geq (n!)^{1/2n}$. 
\end{proof}

%\newpage 
% % % 
\section{Special subvarieties underlying Theorems~\ref{theorem:Bezout} and \ref{theorem:fewnomial}}\label{section:variants}
% % % 

In this section, I indicate how to prove Theorems~\ref{theorem:Bezout} and \ref{theorem:fewnomial}. %- two variants of Theorem~\ref{theorem:main} inspired by \cite{Wooley:symmetric, PW}. 
The essential point is that, for non-degenerate curves in Archimedean fields, the Implicit Function Theorem allows us to upgrade uniform bounds on counting estimates to fat estimates as in Proposition~\ref{prop:main}. For the uniform bounds on counting estimates, we use Bezout's theorem and improvements to it. 

Using Lemma~\ref{lemma:an_uncertain_CS}, the following proposition immediately implies Theorem~\ref{theorem:Bezout}. 
\begin{proposition}\label{prop:Bezout}
%Let $\curve := (T,T^2,\dots,T^n) $ be the moment curve in $\lf = \R$ or $\C$. 
Let $\curve := (\curve_1,\dots,\curve_n) $ be a non-degenerate, polynomial curve in $\lf = \R$ or $\C$. 
%Let $\delta \in \scales(\lf)$. 
%There exists constant $C_{\lf,n}$, depending only on the curve $\curve$, such that 
%For each $n$-tuple $\vof{I} \in \partition_{\delta}^n$, 
We have the bound
\begin{equation}\label{bound:syzygies:Bezout}
%|\syzygies(\delta,\vof{I}; \delta^{n})| 
\S_{\curve} 
\leq 
\big( 2\lceil \ell(\curve) \rceil +1 \big)^{n\cdot \eta_{\lf}} \prod_{i=1}^{n} \deg(\curve_i)
.\end{equation}
\end{proposition}

\begin{proof}[Proof of Proposition~\ref{prop:Bezout}]
Let $\lf$ be $\R$ or $\C$. 
Fix $n \in \N$ to be two or more, fix $\delta \in \scales(\lf)$ and fix $\vof{I} \in \partition_{\delta}^n$. 
Suppose for the moment that $\vof{I}$ is comprised of $n$ distinct intervals. 
The Wronskian $\det\big(\curve'(t) \; \curve''(t) \; \cdots \; \curve^{(n)}(t) \big)$\footnote{Here and below, I am considering $\curve$ as a column vector.} does not vanish for all $t \in \roi$. Consequently, the determinant 
\begin{equation}\label{ineq:NDtoNS}
|\det\big(\curve'(x_1) \; \curve'(x_2) \; \cdots \; \curve'(x_n) \big)| 
\gtrsim \prod_{1 \leq i<j \leq n} |x_i-x_j| 
\end{equation} is non-zero for all $\vof{x} \in \vof{I}$; it can become arbitrarily small if two intervals of $\vof{I}$ are adjacent, but this is not an issue for us. 

Suppose that $\syzygies(\delta,\vof{I};\delta^{n})$ is not empty so that there exists $\vof{x} \in \vof{I}$ and $\vof{y} \in \roi^n$ and $\vof{h} \in \lf^n$ satisfying \eqref{close_points} with $\epsilon = \delta^n$ and 
\[
\sum_{i=1}^{n} \big( \gamma_j(x_i)-\gamma_j(y_i) \big) = h_j 
\quad \text{for} \quad j = 1,\dots,n
.\]
Thus, $|h_j| \leq \delta^n$ for $j = 1,\dots,n$. By a small perturbation, we may assume the strict inequality $|h_j| < \delta^n$ for $j = 1,\dots,n$. 
By Bezout's theorem there are at most $\prod_{i=1}^{n} \deg(\curve_i)$ possibilities for $\vof{y}$.\footnote{Technically, Bezout's theorem gives that there are at most $\prod_{i=1}^{n} \deg(\curve_i)$ non-singular components. There are no singular components since the curve is non-degenerate. Since there are $n$ equations, this means the components are 0-dimensional. In other words, they are points.} 
Using the nonsingularity estimate \eqref{ineq:NDtoNS}, the {Implicit Function Theorem} implies that for any such $\vof{y}$, there exists an open set $U_{\vof{y}}$ in $\lf^n$ upon which \eqref{close_points} holds for some $\epsilon > 0$.  The side-lengths of $U_{\vof{y}}$ are $\leq \ell(\curve) \delta$. 
Since $\R$ and $\C$ are connected, this forces $\vof{J}$ to be one of the $\leq \prod_{i=1}^{n} \deg(\curve_i)$ possible intervals containing $\vof{y}$ or one of each such interval's $\lceil \ell(\curve) \rceil$ neighbors to the left or right (and up or down when $\lf=\C$) for a total of at most $\big( 2\lceil \ell(\curve) \rceil +1 \big)^{n\eta_{\lf}} \prod_{i=1}^{n} \deg(\curve_i)$ possibilities. 

The cases where $\vof{I}$ is not comprised of distinct intervals is proved similarly. The key difference here is to use the fact that a subsystem of the curve has non-vanishing Wronskian on $\roi$. This is sufficient to apply Bezout's theorem to deduce that there are at most $\prod \deg(\curve_i)$ possibilities for each point $\vof{y} \in \vof{J}$. Once again, the Implicit Function Theorem and Lipschitz bound forces these and their neighbors to persist at the fattened level. 
\end{proof}

%See Exercise 7.10 on page 138 in Demeter's book for the passage from the Wronskian to the determinant. This seems to be well-known enough. 

\begin{remark}
Although Bezout's theorem and the Implicit Function Theorem are true over local fields, the above argument is restricted to Archimedean fields $\R$ and $\C$ for two reasons: The first reason is that, in non-Archimedean fields, I do not know if non-vanishing of the Wronskian (that is, non-degeneracy of the curve) implies non-singularity of the curve like in \eqref{ineq:NDtoNS}. 
This first reason is not an issue for the moment curve. 
The second reason is that, for Archimedean fields, I use analytic continuation when applying the Inverse Function Theorem to stop the open sets $U_{\vof{y}}$ from jumping around as $\vof{y}$ varies. This relies on the connectedness of $\R$ and $\C$, but non-Archimedean fields are totally disconnected :/ 
The second reason is an obstruction for this method to apply to the moment curve. This obstruction is overcome in \cite{WH} by making delightful use of how polynomial roots cluster. 
\end{remark}

% Significant improvements for \emph{fewnomials} can be obtained on $\R$. 
Using Lemma~\ref{lemma:an_uncertain_CS}, the following proposition immediately implies Theorem~\ref{theorem:fewnomial}. 
\begin{proposition}\label{prop:fewnomial}
%Let $\curve := (T,T^2,\dots,T^n) $ be the moment curve in $\lf = \R$ or $\C$. 
Let $\curve$ be a non-degenerate, polynomial curve in $\R^n$ such that the total number of monomials appearing in $\curve$ is $M$. 
%Let $\delta \in \scales(\R)$. 
%There exists constant $C_{\curve}$, depending only on the curve $\curve$, such that for each $n$-tuple $\vof{I} \in \partition_{\delta}^n$, we have the bound
\begin{equation}\label{bound:syzygies:fewnomial}
%|\syzygies(\delta,\vof{I}; C_{\curve}^n\delta^{n})| 
\S_{\curve} 
\leq \big( 2\lceil \ell(\curve) \rceil +1 \big)^{n} 2^{M(M-1)/2} (n+1)^{M}
.\end{equation}
\end{proposition}

The proof of the Proposition~\ref{prop:fewnomial} is almost identical to the proof of Proposition~\ref{prop:Bezout}, but utilizes improvements to Bezout's theorem in $\R$ from \cite{Khovanskii:transcendental, Khovanskii:fewnomials} which says that, over $\R$, one may improve the upper bound in Bezout's theorem to $2^{M(M-1)/2} (n+1)^{M}$. I leave remainder of the proof of Proposition~\ref{prop:fewnomial} to the reader. 

%Additionally, with only a minor modification, the proof generalizes to non-degenerate, Pfaff curves in $\R^n$. 
%Let $\curve$ be a non-degenerate, Pfaff curve in $\R^n$ such that the total number of monomials appearing in $\curve$ is $M$, $d_i$ is the degree of a realization of $\curve_i$, and $m$ is the maxima of the $d_i$'s. 
%Let $\delta \in \scales(\R)$. 
%There exists constant $C_{\curve}$, depending only on the curve $\curve$, such that for each $n$-tuple $\vof{I} \in \partition_{\delta}^n$, we have the bound
%\begin{equation}\label{bound:syzygies:Bezout}
%%|\syzygies(\delta,\vof{I}; C_{\curve}^n\delta^{n})| 
%\S_{\curve} 
%\leq \big( 2\lceil \ell(\curve) \rceil +1 \big)^{n} 2^{M(M-1)/2} \Big[ \prod_{i=1}^n d_i \big( mM-1+\sum_{i=1}^n(d_i-1) \big)\Big]^{M}
%.\end{equation}
%\end{proposition}

% % % 
\section{On Superorthogonality}\label{section:SO}
% % % 

In this section I assume that the reader is familiar with \cite{newSO}, and I rephrase some terminology from therein. 
Suppose that $(f_{j_1},\dots,f_{j_{2n}})$ is a tuple of functions from a finite collection of functions $\{f_j\}_{j \in J}$. 
Call a $2n$-tuple of functions $(f_{j_1},\dots,f_{j_{2n}})$ `$2n$-superorthogonal' if 
\begin{equation}\label{eq:SO}
\int f_{j_1}\overline{f_{j_2}} \cdots f_{j_{2n-1}}\overline{f_{j_{2n}}} = 0 
.\end{equation}
Define the ensuing `Types' as subcollections of $J^{2n}$ satisfying the following properties: 
\begin{itemize}
\item Type I*: 
If $(j_1,j_3,\dots,j_{2n-1})$ is not a permutation of $(j_2,j_4,\dots,j_{2n})$, then $(f_{j_1},\dots,f_{j_{2n}})$ is $2n$-superorthogonal. 
% This gives direct and converse inequalities.

\item Type I: 
If some value $j_k$ appears an odd number of times in the $2n$-tuple $(j_1,\dots,j_{2n})$, then $(f_{j_1},\dots,f_{j_{2n}})$ is $2n$-superorthogonal. 
% This gives a direct inequality.

\item Type II: 
If some value $j_k$ appears exactly once in the $2n$-tuple $(j_1,\dots,j_{2n})$, then $(f_{j_1},\dots,f_{j_{2n}})$ is $2n$-superorthogonal. 
% This gives direct linear and multilinear inequalities.

\item Type III: 
If some value $j_k$ appears exactly once and is strictly greater than all other values appearing in the $2n$-tuple $(j_1,\dots,j_{2n})$, then $(f_{j_1},\dots,f_{j_{2n}})$ is $2n$-superorthogonal. 
% With two additional properties, this gives direct and converse inequalities.

\item Type IV: 
If all values of the $2n$-tuple $(j_1,\dots,j_{2n})$ are distinct, then $(f_{j_1},\dots,f_{j_{2n}})$ is $2n$-superorthogonal. 
% This gives a direct inequality.
\end{itemize}
Each successive Type implies the next in the sense that if a subcollection is Type I*, then it is also Type I, and so on. 

For this section, take $\curve := (T,T^2,\dots,T^n)$ to be the moment curve. Since the moment curve is non-degenerate and has degree $n$, the proof of Proposition~\ref{prop:Bezout} reveals that $\syzygies(\delta,\vof{I}; \delta^{n})$ is the set of permutations of $\vof{I}$ along with their neighbors. 
This gives `Type I$^*$ almost superorthogonality' as described in \cite[Section~6]{SO} and \cite[Section~6]{newSO}. The `almost' refers to the neighbors. 

In contrast, the proof of Theorem~\ref{theorem:main} allows for tuples which are not Type I*, I, II, or III. To be precise, let us consider superorthogonality for 6-tuples of an extension operator for the moment curve $(T,T^2,T^3)$. 
Let $f$ be a function and define the functions $f_j := E_{I_j}f$ where $I_j$ is an interval in a fixed, sufficiently fine partition. 
Let us regard $j$ simply as integers in, say, $[1,N]$ for some large $N$. In the proof of Proposition~\ref{prop:main} (and consequently in Theorem~\ref{theorem:main}), I could have $j_1=2$ and $j_2=j_3=j_4=j_5=j_6=1$ potentially give 
\[
\int f_{j_1}\overline{f_{j_2}}f_{j_3}\overline{f_{j_4}}f_{j_5}\overline{f_{j_6}} 
\neq 0
.\]
In other words, I do not prove that \eqref{eq:SO} holds for the 6-tuple $(2,1,1,1,1,1)$. 
Type I* fails because $(2,1,1)$ is not a permutation of $(1,1,1)$. 
Type I fails because 1 appears an odd number of times (as does 2). 
Type II fails because 2 appears as an index precisely once. 
Type III fails because 2 is the unique largest index. 

Curiously, the proof of Proposition~\ref{prop:main} crucially uses the property that \emph{the tuple $(j_2,\dots,j_{2n})$ is a subset of the tuple $(j_1,\dots,j_{2n-1})$} when interpreting both tuples as sets (this means ignoring multiplicity and order of values arising in the tuples). Obviously, this satisfies Type IV superorthogonality. As a consequence, one may bypass Lemma~\ref{lemma:an_uncertain_CS} and reformulate Proposition~\ref{prop:main} to use Theorem~1 of \cite{newSO} to deduce, when the field is Archimedean and of an appropriate characteristic, that $\Hconstant_\curve(W) \leq 2^{1/2}(2n!-1)^{1/2}$. 

Finally, the underlying symmetry of \eqref{inequality:GN} implies Type II superorthogonality which gives a better constant. 
% ; see Section~2 of \cite{SO} (specifically, pages 7101-7102) for more on this constant. 
% There are several improvements to \eqref{bound:syzygies} using some simple combinatorics. Such combinatorics decreases $n^n$ to $\max_{m=1,\dots,n} \big\{ n(n-1)\cdots(n-m+1) \cdot m^{n-m} \big\}$, or to Stirling numbers of the second kind. 
When $n=2$ or 3, one easily refines the bounds on $\Hconstant_\curve(W)$ to $2!$ and $3!$ respectively. 
Unfortunately, the best bound deducible from this method exceeds $n!$ when $n \geq 4$. 
%

% To be precise I allow for the following situation. Suppose that I have a function $f$ and I am considering SO for 6-tuples. So I have 6-tuples $(f_{j_1},\dots,f_{j_6})$ for indices $j_1,\dots,j_6$ which I think of as representing intervals coming from an extension operator, but equivalently (for this discussion) they are just integers in, say, $[1,N]$ for some large $N$. In the proof of Proposition~\ref{prop:main} (and consequently in Theorem~\ref{theorem:main}), I could have $j_1=2$ and $j_2=j_3=j_4=j_5=j_6=1$ potentially give that $\int f_{j_1} \cdots f_{j_6} \neq 0$. As far as I understand your synopsis of SO, this is not: 
% \begin{itemize}
% \item[Type I*] SO because $(j_1,j_2,j_3)=(2,1,1)$ is not a permutation of $(j_4,j_5,j_6)=(1,1,1)$
% \item[Type I] SO because 1 appears an odd number of times (as does 2)
% \item[Type II] SO because 2 appears as an index precisely once
% \item[Type III] SO because 2 is the unique largest index
% \end{itemize}
% However, it is true for me that whenever $j_1,\dots,j_6$ are all distinct, then $\int f_{j_1} \cdots f_{j_6} = 0$. Indeed, that is the point of the proof of Proposition 5 in my work (linked above).

%\newpage 
% % % 
%\input{KH_biblio.tex}
% % % 

\bibliographystyle{amsalpha}
% \bibliographystyle{plain}
%\bibliography{references_lp_improving}

\end{document}